\documentclass[12pt]{article}
\usepackage{latexsym,amsmath,amssymb,, amsthm}

\newif\ifdviwin

\usepackage[latin1]{inputenc}
\usepackage[english]{babel}
\usepackage{indentfirst}
\usepackage[mathscr]{eucal}
\usepackage{amssymb,amsmath,amsfonts}
\usepackage{graphicx}

\newif\ifdviwin

\dviwintrue

\let\8=\infty \let\0=\emptyset  
  \let\por=\times

\let\con=\overline

\let\ro=\rho

\def\flecha{\rightarrow}
\def\esiz{\langle}
\def\esde{\rangle}

\def\cte.{\mathop{\rm cte.}\nolimits}

\def\Re{\mathop{\rm Re }\nolimits}

\def\cosh{\mathop{\rm cosh }\nolimits}

\def\L{\mathbb{L}}

\def\R{\mathbb{R}}

\def\C{\mathbb{C}}

\def\S{\mathbb{S}}

 \newtheorem{defi}{Definition}[section] \newtheorem{teo}[defi]{Theorem}
 
 \newtheorem{cor}[defi]{Corollary}
 
 \newtheorem{eje}[defi]{Example}
 \newtheorem{remark}{Remark}

\DeclareGraphicsExtensions{.png,.jpg,.pdf,.mps,.gif,.bmp}

\numberwithin{equation}{section}

\begin{document}
\mbox{}\vspace{0.4cm}\mbox{}

\begin{center}
\rule{15.2cm}{1.5pt} \vspace{0.5cm}

{\Large \bf The Cauchy problem for indefinite improper affine spheres and their Hessian equation \footnote{Research partially supported by Ministerio de Educaci\'on y Ciencia Grant No. MTM2010-19821 and Junta de Andaluc\'{\i}a Grant Nos. FQM-325 and P09-FQM-5088}} \\
\vspace{0.5cm}
{\large Francisco Milán}\\
\vspace{0.3cm} \rule{15.2cm}{1.5pt}
\end{center}
Departamento de Geometr\'{\i}a y Topolog\'{\i}a, Universidad de Granada, E-18071 Granada, Spain. \\
e-mail: milan@ugr.es
\vspace{0.2cm}



\vspace{0.3cm}

\begin{abstract}
We give a conformal representation for indefinite improper affine spheres which solve the Cauchy problem for their Hessian equation. As consequences, we can characterize their geodesics and obtain a generalized symmetry principle. Then, we classify the helicoidal indefinite improper affine spheres and find a new family with complete non flat affine metric. Moreover, we present interesting examples with singular curves and isolated singularities.
\end{abstract}

\section{Introduction}
It is well-known the good interplay of differential equations and differential geometry. For instance, the theory of Monge-Amp\`{e}re equations has many geometric applications like surfaces of prescribed Gauss curvature or affine spheres. Conversely, the properties of these surfaces play an important role in the development of geometric methods for the study of their PDEs, see the recent works of Geometric Analysis \cite{L, LJSX, TW}.

Here, we consider the classical Monge-Amp\`{e}re equation $f_{xx} f_{yy} - f_{xy}^2 = \pm 1$ and the improper affine sphere, given locally by the graph of a solution $f(x, y)$, see \cite{LSZ, NS}. Of course, the situation changes completely if we take $+1$ or $-1$.

In the first case, we have the elliptic Hessian $+1$ equation and there are strong results about it. Actually, J\"{o}rgens \cite{J1, J2} proved that the revolution surfaces provide the only entire solutions with at most an isolated singularity to this equation. Then, as geometric consequence, the elliptic paraboloid is the only global example and the unique improper affine sphere with complete definite affine metric, see \cite{C1, P, CY}. Conversely, as the affine conormal of an improper affine sphere is harmonic with respect to the affine metric, we obtained in \cite{FMM1, FMM2} a conformal representation with many analytic and geometric applications, see \cite{ACG, CL, FMM3, GMM, GMMi, Gu, LX}.

The non elliptic case is more complicated and we can not expect a version of the above global results. In fact, $f(x, y) = x y + g(x)$ is a solution for any function $g$ and there are many ruled improper affine spheres, with complete flat affine metric, see \cite{MR}. Then, in order to find complete non flat examples, we give a conformal representation for indefinite improper affine spheres, which solve the Cauchy problem for the Hessian $-1$ equation and has many analytic and geometric applications. Also, in computer vision, it can give new algorithms for computing the area distance $f$ of a convex curve, see \cite{CST1}.

Hence, we are going to solve the following affine Bj\"{o}rling-Cauchy Problem:
\begin{quote}
Let $\alpha:I\longrightarrow \R^3 $ be a regular analytic curve such that $\alpha' \times \alpha'' \neq 0$. Find all indefinite improper affine spheres containing $\alpha$.
\end{quote}

This problem has been inspired by the classical Bj\"{o}rling problem for minimal surfaces in $\R^3$, solved by Schwarz in 1890 and generalized to other families of surfaces in \cite{ACG, ACM, B, CDM, DHKW, GMi}.

Thus, after some preliminaries, we study in Section 3 necessary and sufficient conditions for the existence and uniqueness of solutions to the above problem. We also construct the indefinite improper affine sphere in terms of the curve $\alpha $ and give new conformal representations. As first consequence, we solve explicitly the Cauchy problem for the Hessian $-1$ equation.

Section 4 is devoted to characterize when a  curve in $\R^3$ can be geodesic or pre-geodesic of an indefinite improper affine sphere and other applications, which motivate the study of symmetries. Then, in Section 5, we obtain all the helicoidal indefinite improper affine spheres, that is, they are invariant under a one-parameter group of equiaffine transformations. Here, we find our examples with complete non flat affine metric. Also, we see that the revolution examples do not provide entire solutions with at most an isolated singularity to the Hessian $-1$ equation.

Finally, in Section 6, we extend our conformal representation to indefinite improper affine maps with a singular set containing a prescribed analytic curve and propose it to continue the study of singularities, see \cite{CST2, IM, M, Mi, N}.

\section{Preliminaries}

Consider $\psi:\Sigma \flecha \R^3$ an indefinite improper affine sphere, that is, an immersion with constant affine normal $\xi$. Then, see \cite{LSZ, NS},
up to an equiaffine transformation, one has $\xi = (0, 0, 1)$ and $\psi$ can be locally seen as the graph of a solution
$f(x, y)$ of the Hessian $- 1$ equation
\begin{equation}
f_{xx} f_{yy} - f_{xy}^2 = - 1.  \label{hes}
\end{equation}

In such case, the affine conormal $N$ and the indefinite affine metric $h$ of $\psi$ are given by
\begin{eqnarray}
N & = & (- f_x, - f_y, 1), \nonumber \\
\label{conxy} \\
h & = & f_{xx} dx^2 + 2 f_{xy} dx dy + f_{yy} dy^2 \nonumber
\end{eqnarray}
and (\ref{hes}) is equivalent to
\begin{eqnarray}
\sqrt{h(\psi_x, \psi_y)^2 - h(\psi_x, \psi_x) h(\psi_y, \psi_y)} = [\psi_x, \psi_y, \xi] = - [N_x, N_y, N], \label{vol}
\end{eqnarray}
that is, the volume element of $h$ coincides with the determinant $[.,., \xi]$.

We also observe that $h = - \esiz dN, d\psi \esde$ and $N$ is determined by
\begin{equation}
\esiz N,\xi \esde = 1, \qquad \esiz N, d\psi\esde = 0, \label{proconormal}
\end{equation}
with the standard inner product $\esiz \  , \ \esde$ in $\R^3$. Moreover, from (\ref{hes}) and (\ref{conxy}), one can obtain
\begin{equation}
\Delta_h N = 0, \nonumber
\end{equation}
where $\Delta_h$ is the Laplace-Beltrami operator associated to $h$. Consequently, if $s$ and $t$ are conformal coordinates for $h$, then $N$ satisfies the homogeneous wave equation
\begin{equation}
N_{tt} - N_{ss} = 0. \nonumber
\end{equation}

Then, following \cite{CDM, IT}, we write the above affine invariants with the split-complex numbers
\begin{equation}
\C' = \{ z = s + j t : s, t \in \R, j^2 = 1, 1 j = j 1\} \nonumber
\end{equation}
and the partial derivatives
\begin{eqnarray}
\psi_z  = \frac{1}{2} \left( \psi_s + j \psi_t \right) , \qquad \psi_{\overline{z}}  = \frac{1}{2} \left( \psi_s - j \psi_t \right), \nonumber
\end{eqnarray}
with the usual notation $\overline{z} = s - j t$, $Re(z) = s$ and $Im(z) = t$.

In particular, see \cite{C2, FMM2} for the definite case with complex numbers, from (\ref{vol}) and  (\ref{proconormal}), we have
\begin{eqnarray}
h = 2 \ro \ dz \ d\overline{z}, \qquad \ro = \esiz N, \psi_{z\overline{z}} \esde = - j [\psi_z,\psi_{\overline{z}}, \xi] = j [N, N_z, N_{\overline{z}}] > 0 \label{ro}
\end{eqnarray}
and
\begin{eqnarray}
\xi  =  \frac{j}{\ro} N_z \por N_{\con{z}}, \qquad
N  =  \frac{- j}{\ro} \psi_z \por \psi_{\con{z}},  \label{conz}
\end{eqnarray}
where by $\por$ we denote the cross product in $\C'^3$. We also get
\begin{eqnarray}
\psi_z  =  j N_z \por N, \qquad
 N_z  =  j \xi \por \psi_z \label{isotermos}
\end{eqnarray}
and
\begin{eqnarray}
\psi_{z{\overline{z}}}  =  \rho \xi, \qquad
 N_{z{\overline{z}}}  =  0. \label{laplaz}
\end{eqnarray}

\begin{remark} \label{r1}
In general, the split-holomorphic curve $N_z : \Sigma \flecha \C'^3$ can be integrated locally. However, from (\ref{isotermos}), we have a global split-holomorphic curve $\Phi: \Sigma \flecha \C'^3$, such that
\begin{eqnarray}
N =  \Phi + \con{\Phi}, \qquad j \xi \por \psi =  \Phi - \con{\Phi}. \label{conormal}
\end{eqnarray}

Conversely, from (\ref{isotermos}), we recover the indefinite improper affine sphere $\psi$ with its affine conormal and the conformal  class of the affine metric as
\begin{equation}
\psi = 2\Re \int j N_z \por N dz,  \label{Lelieuvre}
\end{equation}
which, along with (\ref{proconormal}), (\ref{ro}) and (\ref{conormal}), give that $\psi$ is uniquely determined, up to a real translation, by a split-holomorphic curve $\Phi$ satisfying
\begin{eqnarray}
2 \esiz \Phi, \xi \esde = 1, \qquad j \left[\Phi + \con{\Phi}, \Phi_z, \con{\Phi_z}\ \right] > 0, \label{condicion}
\end{eqnarray}
for a non zero vector $\xi$. To be precise,
\begin{equation}
\psi = - 2\Re \int j \left(\Phi + \con{\Phi}\right)\por \Phi_z dz. \label{inmer}
\end{equation}
\end{remark}

In fact, from (\ref{condicion}) and (\ref{inmer}), we obtain
\begin{equation*}
\psi_z \por \psi_{\con{z}} = - \left[\ \Phi + \con{\Phi}, \Phi_z, \con{\Phi_z}\ \right] (\Phi + \con{\Phi}),
\end{equation*}
\begin{equation*}
\psi_{z\con{z}} =  j \Phi_z\por\con{\Phi_{z}} = j \left[\ \Phi + \con{\Phi}, \Phi_z, \con{\Phi_z}\ \right] \xi
\end{equation*}
and
\begin{equation*}
\esiz \Phi + \con{\Phi}, \psi_{z\overline{z}} \esde = j \left[\ \Phi + \con{\Phi}, \Phi_z, \con{\Phi_z}\ \right] = - j [\psi_z,\psi_{\overline{z}}, \xi] > 0.
\end{equation*}
Thus, from (\ref{ro}), (\ref{conz}) and (\ref{laplaz}), the immersion $\psi$ has constant affine normal $\xi$ and it is an indefinite improper affine sphere, with affine conormal $\Phi + \con{\Phi}$.

\section{The indefinite Affine Bj\"{o}rling problem}
The above conformal representation in terms of split-holomorphic data, let us
solve the \textit{"affine Bj\"{o}rling problem"} of finding indefinite improper affine spheres containing a regular analytic curve
with a prescribed affine conormal along it.

Now, in order to deduce the necessary conditions, we take an indefinite improper affine sphere $\psi:\Sigma \flecha \R^3$
with constant affine normal $\xi$ and affine conormal $N$. Let $I$ be an interval and $\beta : I \flecha \Sigma$
a regular analytic curve. If $\alpha = \psi \circ \beta$ and $U = N \circ \beta$, then (\ref{proconormal}) gives
\begin{equation}
\left. \begin{array}{rll}
0 &=& \esiz \alpha', U\esde, \\
1 &=& \esiz \xi, U\esde,\\
\lambda & = & \esiz \alpha'', U\esde = - \esiz \alpha', U'\esde,
\end{array}\right\} \label{admisible}
\end{equation}
where by prime we indicate derivation respect to $s$, for all $s \in I$.

\begin{remark}
The problem for $\lambda = 0$ was solved by Blaschke in \cite{B} with asymptotic parameters. In particular, he classified the family of ruled improper affine spheres, see also \cite{MM, MR}.

So, we will consider $\lambda > 0$ and a conformal parameter $z = s + j t$, (or $\lambda < 0$ and $jz$). Here, it is clear that $\alpha' \times \alpha''$ does not vanish anywhere.
\end{remark}

Motivated by these facts, we say that a pair of regular analytic curves $\alpha, U : I \flecha \R^3$ is \textit{admissible} for a non zero vector $\xi$ if there is an analytic positive function $\lambda :I \flecha \R^+$ such that all the  equations in (\ref{admisible}) hold on $I$.

Conversely, we prove that the above conditions are also sufficient to find a unique solution of the \textit{"affine Bj\"{o}rling problem"}.

\begin{teo} \label{t1} Let $\alpha, U : I \flecha \R^3$ be an admissible pair of curves for a non zero vector $\xi$. Then there exists  a unique indefinite improper affine sphere $\psi$ containing $\alpha(I)$ with affine conormal $U(s)$ at $\alpha(s)$ for all $s \in I$ and affine normal $\xi$.
\end{teo}

\begin{proof}
If we assume that $\psi$ is a solution, by the inverse function theorem, there exists a conformal parameter $z = s + j t$ for the affine metric of $\psi$, defined in a split-complex domain containing $I$.

Then, from (\ref{conormal}), $\Phi = \frac{1}{2} (N + j \xi \por \psi)$ is a split-holomorphic curve such that
\begin{eqnarray*}
\Phi(s) = \frac{1}{2} (U(s) + j \xi \por \alpha(s)), \qquad s \in I
\end{eqnarray*}
and by analyticity we have that
\begin{eqnarray}
\Phi(z) = \frac{1}{2} (U(z) + j \xi \por \alpha(z)), \qquad z \in \Omega \subset \C', \label{phi}
\end{eqnarray}
in a domain $\Omega$ containing $I$, where the split-holomorphic extensions of $U$ and $\alpha$ exist.

Thus, from (\ref{conormal}), (\ref{Lelieuvre}) and (\ref{phi}), the immersion $\psi$ is determined by $U$ and $\alpha$, which proves the uniqueness.

For the existence, we consider the above split-holomorphic curve. Now, from (\ref{admisible}) and (\ref{phi}), we have
\begin{eqnarray*}
2 \esiz \Phi, \xi \esde = \esiz U, \xi \esde = 1
\end{eqnarray*}
in $\Omega$ (by analyticity) and
\begin{eqnarray*}
 j \left[\ \Phi + \con{\Phi},
  \Phi_z, \con{\Phi_z}\ \right] & = & \frac{j}{4}\left[U, U' + j \xi \por \alpha',
  U' - j \xi \por \alpha' \right]\nonumber\\& = & -\frac{1}{2} \esiz U \por
 U', \xi \por \alpha' \esde = \frac{\lambda}{2} > 0
\end{eqnarray*}
along $\alpha$. Then, from (\ref{condicion}) and (\ref{inmer}),
\begin{equation}
\psi =  \alpha(s_0) - 2 \Re \int_{s_0}^z j(\Phi + \con{\Phi})\por
\Phi_\zeta d\zeta, \qquad s_0 \in I, \label{maximm}
\end{equation}
is an indefinite improper affine sphere, in a neighborhood of $I$, with affine normal $\xi$ and affine conormal $\Phi + \con{\Phi}$
which is an extension of  $U(I)$. Finally, along $\alpha$,
\begin{eqnarray*}
 \psi_{z} & = &  - j (\Phi + \con{\Phi}) \por \Phi_z =  \frac{- j}{2} \ U \por U' - \frac{1}{2} \ U \por \left( \xi \por \alpha'\right) \\
  & = &  \frac{1}{2} \ \alpha' -  \frac{j}{2}\
  U \por U'
\end{eqnarray*}
and the immersion contains the curve $\alpha(I)$.
\end{proof}

From now on, without loss of generality, we will assume that, up an equiaffine transformation, the affine normal is $\xi = (0, 0, 1)$. Thus, the Theorem \ref{t1} gives a conformal representation in terms of two split-holomorphic functions.

\begin{teo} \label{t2}
Let $\alpha, U : I \flecha \R^3$ be an admissible pair of curves for $\xi = (0, 0, 1)$. Then the unique indefinite improper affine sphere $\psi = (\psi_1, \psi_2, \psi_3)$ containing $\alpha(I)$ with affine conormal $U(s)$ at $\alpha(s)$ for all $s \in I$ is given by
\begin{eqnarray*}
N & = & Re(2 \Phi_1, 2 \Phi_2, 1),\\
(\psi_1, \psi_2) & = & Im(2 \Phi_2, -2 \Phi_1),\\
- \psi_{3z} & = & \psi_{1z} N_1 + \psi_{2z} N_2,
\end{eqnarray*}
with $2 (\Phi_1, \Phi_2) = (U_1 - j \alpha_2, U_2 + j \alpha_1)$ in a neighborhood of $I$ in $\C'$.
\end{teo}

\begin{proof}
It is clear from (\ref{conormal}), (\ref{phi}) and (\ref{maximm}).
\end{proof}

Finally, our conformal representation let us to solve the Cauchy problem for the Hessian $- 1$ equation
\begin{equation}
\left\{ \begin{array}{lll}
f_{xx} f_{yy} - f_{xy}^2  =  - 1, \\
f(x, 0)  =  a(x), \qquad a''(x) > 0,\\
f_y(x, 0)  =  b(x),
\end{array} \right. \label{Cauchy}
\end{equation}
where $a, b$ are two analytic functions defined on an interval $I$. Of course, if $a''(x) < 0$, we can change $f$ by $-f$.

\begin{teo} \label{t3}
There exits a unique solution to the Cauchy problem (\ref{Cauchy}) in a neighborhood of $I$ in $\R^2$ given by
\begin{equation*}
f(x, y) =  a(x_0) + \frac{1}{2} \Re \int_{x_0}^z \left( (a'(\zeta) + \overline{a'(\zeta)})(1 - j b'(\zeta)) + a''(\zeta)(j b(\zeta) + j \overline{b(\zeta)} - \zeta + \overline{\zeta}) \right) d\zeta,
\end{equation*}
with
\begin{equation*}
x(z) = Re(z) - Im(b(z)), \qquad y(z) = Im(a'(z)),
\end{equation*}
where $a(z)$ and $b(z)$ are the split-holomorphic extensions of $a(x)$ and $b(x)$, respectively, and $x_0 \in I$.
\end{teo}

\begin{proof}
The result follows from Theorem \ref{t2}, since the problem (\ref{Cauchy}) is equivalent to find the indefinite improper affine sphere with the admissible pair of curves
\begin{equation*}
\{\alpha(s) = (s, 0, a(s)), \ U(s) = (-a'(s), -b(s), 1)\},
\end{equation*}
for $s \in I$.
\end{proof}

\section{Some consequences}

First, we use that for an admissible pair $\{\alpha ,U\}$, the curve $U$ is determined by $\alpha$ and the affine metric $h$ along $\alpha$, when $[\alpha', \alpha'', \xi] \neq 0$ on $I$.

In the case $[\alpha', \alpha'', \xi] \equiv 0$ there is not uniqueness, since $\{\alpha ,U + \mu \xi \times \alpha' \}$ is also an admissible pair, for any analytic function $\mu$.

\begin{teo}\label{t4}
Let $\alpha : I \flecha \R^3$ be a regular analytic curve satisfying
\begin{equation*}
[\alpha'(s), \alpha''(s), \xi] \neq 0
\end{equation*}
for all $s \in I$. Then, given a positive analytic function $\lambda:I\flecha \R^+$, there exists a
unique indefinite improper affine sphere $\psi$ containing the curve $\alpha(I)$, with $h(\alpha',\alpha') = \lambda$ on $I$.

Moreover,  the immersion $\psi$ can be
written as (\ref{maximm}) in a neighborhood of $I$ in $\C'$, with
\begin{equation*}
\Phi = \frac{\alpha_z \por \left(\alpha_{zz} - \lambda \xi\right)} {2[\alpha_z, \alpha_{zz}, \xi]} + \frac{j}{2} \ \xi \por \alpha.  \label{curvaholo3}
\end{equation*}
\end{teo}
\begin{proof}
From (\ref{admisible}), there
is  a unique $U : I \flecha \R^3$,
\begin{equation}
U = \frac{\alpha' \por \left( \alpha'' - \lambda \xi \right)}{[\alpha', \alpha'', \xi]}, \label{conormal3}
\end{equation}
such that $\{\alpha ,U\}$ is an admissible pair of curves for $\xi$ and the result follows from Theorem \ref{t1}.
\end{proof}

In particular, we obtain that the revolution examples can be recovered with one of their circles and the affine metric along it, (see Section 5 for an exhaustive study).

\begin{cor}\label{c3}
Let $\alpha : \R \flecha \R^3$ be the curve $\alpha(s) = (c \cos(s), c \sin(s), 0)$, with $c > 0$. Then, the unique indefinite improper affine sphere $\psi = (\psi_1, \psi_2, \psi_3)$ containing the circle $\alpha(\R)$ with constant affine metric $m > 0$ on $\alpha(\R)$ has coordinates
\begin{eqnarray*}
\psi_1 & = & \cos(s) \left(c \cos(t) - \frac{m \sin(t)}{c} \right),  \\
\psi_2 & = & \sin(s) \left(c \cos(t) - \frac{m \sin(t)}{c} \right),  \\
\psi_3 & = & -\frac{1}{2} \left(c^2 + \frac{m^2}{c^2} \right) t+ m (\cos(t)^2 - 1) + \frac{1}{4} \left(c^2 - \frac{m^2}{c^2} \right) \sin(2 t).
\end{eqnarray*}
\end{cor}
\begin{proof}
From (\ref{conormal3}), we obtain
\begin{eqnarray*}
U(s) = \left( -\frac{m}{c} \cos(s), -\frac{m}{c} \sin(s), 1 \right)
\end{eqnarray*}
and Theorem \ref{t2} gives the above immersion $\psi$, with the split-holomorphic extensions
\begin{eqnarray*}
\cos(s + jt) =  \frac{ \cos(s + t) + \cos(s - t)}{2} + j \ \frac{\cos(s + t) - \cos(s - t)}{2}
\end{eqnarray*}
and
\begin{eqnarray*}
\sin(s + jt) =  \frac{ \sin(s + t) + \sin(s - t)}{2} + j \ \frac{\sin(s + t) - \sin(s - t)}{2}.
\end{eqnarray*}
\end{proof}

Second, we characterize the curves in $\R^3$ which can be pre-geodesics or geodesics of some indefinite improper affine sphere.

\begin{teo}\label{t5}
Let $\psi:\Sigma \flecha \R^3$ be an indefinite improper affine sphere with affine normal $\xi$ and affine conormal $N$. If
$\beta : I \flecha \Sigma$ is a regular analytic curve, $\alpha = \psi \circ \beta$ and $U = N \circ \beta$,
then $\alpha$ is a pre-geodesic for the affine metric if and only if
\begin{equation} \label{ConPregeo}
[\alpha', \alpha'', \xi] = [U, U', U''] \qquad \mbox{on}\quad I.
\end{equation}
\end{teo}
\begin{proof}
As we have seen in Theorem \ref{t1}, there exists a conformal parameter
$z = s + j t$ for the affine metric $h$, such that $\psi(s,0)=\alpha(s)$.

It is well-known that $\alpha$ is a pre-geodesic if and only if
$\nabla_{\alpha'(s)}\alpha'(s)$ is proportional to $\alpha'(s)$,
where $\nabla$ is the Levi-Civita connection of $h$, or
equivalently,
\[
0 = h\left( \nabla_{\frac{\partial}{\partial s}} \frac{\partial}{\partial s}, \frac{\partial}{\partial t}  \right)
= - \frac{1}{2} \frac{\partial}{\partial t} h \left( \frac{\partial}{\partial s}, \frac{\partial}{\partial s} \right)
= - \frac{\partial}{\partial t}
h \left( \frac{\partial}{\partial z}, \frac{\partial}{\partial \overline{z}} \right)
\]
along $\alpha(s)$. That is, the imaginary part of
\[
\frac{\partial}{\partial z} h \left( \frac{\partial}{\partial z},
\frac{\partial}{\partial \overline{z}} \right)
\]
vanishes identically for all $z=s\in I$ and we conclude, from (\ref{ro}), (\ref{conormal}), (\ref{admisible}) and
(\ref{phi}), since
\begin{eqnarray*}
\frac{\partial}{\partial z} h \left( \frac{\partial}{\partial z},
\frac{\partial}{\partial \overline{z}} \right)&=&
j \left[N,N_{zz},N_{\overline{z}}\right] =\frac{j}{4} \left[ U,
U'' + j \xi \times \alpha'', U' - j \xi \times \alpha' \right]\\
& = & \frac{1}{4}\left([U, \xi \times \alpha'',U'] - [U,U'', \xi \times \alpha'] \right) \\
& + & \frac{j}{4}\left([\alpha', \alpha'', \xi] - [U, U', U'']\right)
\end{eqnarray*}
along $\alpha$.
\end{proof}

As a pre-geodesic $\alpha$ is a geodesic for $h$ if and only if $h(\alpha',\alpha')$ is constant on $I$, Theorems \ref{t1},  \ref{t4} and \ref{t5} give

\begin{cor}\label{c4}
A regular analytic curve $\alpha : I \flecha \R^3$, with $[\alpha', \alpha'', \xi] \neq 0$ on $I$, is the geodesic of some indefinite improper affine sphere if and
only if
\begin{equation*}
U = \frac{\alpha' \por \left( \alpha'' - m \xi \right)}{[\alpha', \alpha'', \xi]}
\end{equation*}
verifies (\ref{ConPregeo}), for some constant $m > 0$.
\end{cor}

\begin{remark}
From \cite{ACG}, we know that the curve $\alpha(s) = (c \cos(s), c \sin(s), 0)$ can not be a geodesic of a definite improper affine sphere. However, from Corollaries \ref{c3} and \ref{c4}, $\alpha$ is a geodesic of the revolution indefinite improper affine sphere with $m = c^2 > 0$.
\end{remark}

\begin{remark}
Instead, for an admissible pair $\{\alpha ,U\}$, with $m \alpha'' = m^2 \xi - \esiz \alpha'', U' \esde \alpha'$, the curve $\alpha$ is a geodesic of the family of indefinite improper affine spheres generated by $\widetilde{U} = U + \mu \xi \times \alpha'$, for any analytic function $\mu$, such that
\begin{equation*}
[\xi, \widetilde{U}', \widetilde{U}''] = [\widetilde{U}, \widetilde{U}', \widetilde{U}''] =0.
\end{equation*}
For instance, $\alpha(s) = (s, 0, m s^2/2)$ has $\alpha''(s) = m \xi$, $\widetilde{U}(s) = (- m s, \mu(s), 1)$ and the above condition is $\mu'' = 0$.
\end{remark}

\section{Helicoidal indefinite improper affine spheres}

In Corollary \ref{c3}, we see that any symmetry in the circle is also a symmetry in the (revolution) indefinite improper affine sphere which generates.

Motivated by this fact, we consider the equiaffine transformation $T:\R^3 \flecha \R^3$ given by
\begin{equation*}
T(v) = A v + v_0, \; \; v \in \R^3,
\end{equation*}
where $A\in\S \L(3, \R) $  and $v_0 \in \R^3$. We will say that $T$ is a symmetry of an admissible pair $\{\alpha ,U\}$ if $A \xi = \xi$ and there exists an
analytic diffeomorphism $\Gamma : I \flecha I$ such that $\alpha \circ \Gamma = T
\circ \alpha$ and $U \circ \Gamma = (A^t)^{-1} U$.

Thus, the following result is an extension of  Theorem 4.2 in
\cite{ACM} and it can be proved analogously to the corresponding
one in \cite{GMi}.

\begin{teo}\label{gsp}(Generalized symmetry principle).
Any symmetry of an admissible pair induces a global symmetry of
the indefinite improper affine sphere generated by it.
\end{teo}

As application of Theorems \ref{t1} and \ref{gsp}, we shall obtain the indefinite improper affine spheres which are invariant under a one-parametric group of equiaffine transformations $\{T^s \; : \; s \in \R\}$.

On the one hand, we have the well-known ruled improper affine spheres. On the other, from \cite{ACG} and \cite{AMM}, we must study the orbits $\alpha_p$ of a fixed point $p \in \R^3$ under the following three groups:
\begin{eqnarray}
T_{1a}^s(v)\ &  = &
\left( \begin{array}{ccc}
1 & 0 & 0  \\
s & 1 & 0  \\
\frac{a s^2}{2} & as & 1
\end{array} \right) v + \left( \begin{array}{c} s \\ \frac{s^2}{2} \\ \frac{a s^3}{6} \end{array} \right) \label{g1} \\
T_{2a}^s(v)\  & = &
\left( \begin{array}{ccc}
\cos(s) & \sin(s) & 0 \\
-\sin(s)& \cos(s) & 0 \\
0 & 0 & 1
\end{array} \right) v + \left( \begin{array}{c} 0 \\ 0 \\ a s \end{array} \right) \label{g2} \\
T_{3a}^s(v)\  & = &
\left( \begin{array}{ccc}
e^s & 0 & 0 \\
0 & e^{-s} & 0 \\
0 & 0 & 1
\end{array} \right) v + \left( \begin{array}{c} 0 \\ 0 \\ a s \end{array} \right) \label{g3}
\end{eqnarray}
where $a \in \R$. We will use that the affine metric is constant along $\alpha_p$.

\begin{figure}
\begin{center}
\includegraphics[width=0.30\textwidth]{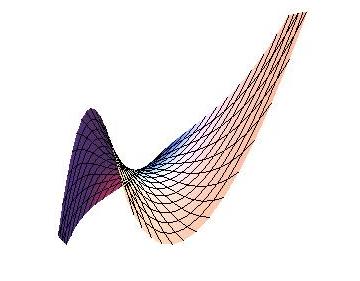} \qquad
\includegraphics[width=0.26\textwidth]{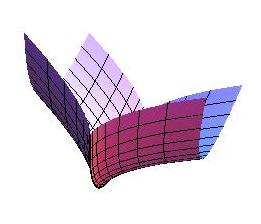} \\
\caption{(\ref{g1})-invariant indefinite improper affine spheres.}
\end{center}
\end{figure}

\subsection{Group (\ref{g1})}

In this case, every orbit go through a point $p = (0, b, c)$ and can be written as
\begin{equation*}
\alpha_p(s) = T_{1a}^s(p) = \left(s, b + \frac{s^2}{2},  a b s + c + a \frac{s^3}{6} \right).
\end{equation*}
Thus, $[\alpha_p', \alpha_p'', \xi] = 1$ on $\R$ and, for a constant $m > 0$, Theorem \ref{t4} gives
\begin{equation*}
U_p(s) = \left( -a b - m s + a \frac{s^2}{2}, m - a s, 1\right)
\end{equation*}
and the associated immersion $\psi$ has coordinates
\begin{eqnarray*}
\psi_1 & = & s - a t,  \\
\psi_2 & = & b + \frac{s^2}{2} + m t - a s t + \frac{t^2}{2},  \\
\psi_3 & = & c + a b s + a \frac{s^3}{6} - m^2 t - a^2 b t + a m s t - \frac{s^2}{2} a^2 t -
  m t^2 + \frac{t^2}{2} a s - \frac{t^3}{3} + a^2 \frac{t^3}{6}.
\end{eqnarray*}
Now, as the affine metric is
\begin{equation*}
h = [\psi_s, \psi_t, \xi](ds^2 - dt^2) = \left( m + (1 - a^2)t \right) (ds^2 - dt^2),
\end{equation*}
we obtain two different cases, (see Figure 1):
\begin{enumerate}
\item If $a^2 = 1$, then $\psi(\R^2)$ is a Cayley surface, with complete flat affine metric.
\item If $a^2 \neq 1$, then $\psi$ is an immersion when $m \neq (a^2 - 1)t$ and $\psi(\R^2)$ contains the singular curve $\gamma(\R)$ with
\begin{equation*}
\gamma(s) = \psi \left( s, \frac{m}{a^2 - 1} \right), \qquad s \in \R.
\end{equation*}
\end{enumerate}

\begin{figure}
\begin{center}
\includegraphics[width=0.125\textwidth]{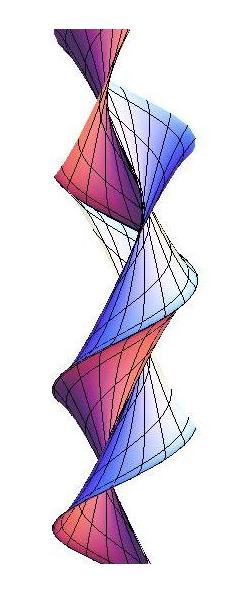} \qquad
\includegraphics[width=0.18\textwidth]{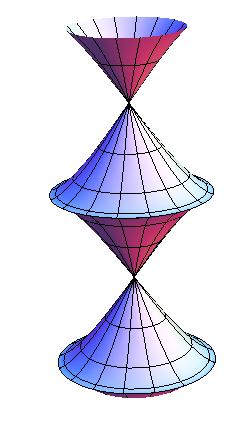} \\
\caption{(\ref{g2})-invariant indefinite improper affine spheres.}
\end{center}
\end{figure}

\subsection{Group (\ref{g2})}

Similarly, we can consider the orbit of a point $p = (c, 0, 0)$, with $c > 0$, that is
\begin{equation*}
\alpha_p(s) = T_{2a}^s(p) = (c \cos(s), c \sin(s), a s)
\end{equation*}
and $[\alpha_p', \alpha_p'', \xi] = c^2$ on $\R$. Then
\begin{equation*}
U_p(s) = \left( \frac{ - m \cos(s) + a \sin(s)}{c}, \frac{ - a \cos(s) - m \sin(s)}{c}, 1 \right)
\end{equation*}
and the associated immersion $\psi$ has coordinates
\begin{eqnarray*}
\psi_1 & = & \cos(s) \left(c \cos(t) - \frac{m \sin(t)}{c} \right) + \frac{a \sin(s) \sin(t)}{c},  \\
\psi_2 & = & \sin(s) \left(c \cos(t) - \frac{m \sin(t)}{c} \right) - \frac{a \cos(s) \sin(t)}{c},  \\
\psi_3 & = & a s - \frac{1}{2} \left(c^2 + \frac{m^2}{c^2} +  \frac{a^2}{c^2} \right) t+ m (\cos(t)^2 - 1) + \frac{1}{4} \left(c^2 - \frac{m^2}{c^2} - \frac{a^2}{c^2} \right) \sin(2 t).
\end{eqnarray*}

Here, the affine metric is
\begin{equation*}
h = \left( m \cos(2 t) - \frac{\left(a^2-c^4+m^2\right) \cos(t) \sin(t)}{c^2} \right) (ds^2 - dt^2)
\end{equation*}
and $\psi(\R^2)$ has singular curves for all $a \in \R$, (see Figure 2).

Moreover, in the revolution case, $a = 0$, there are isolated singularities when $c^2 \cos(t) = m \sin(t)$.

\begin{remark}
In contrast to J\"{o}rgens theorem, the revolution surfaces do not provide entire solutions with at most an isolated singularity to the Hessian $-1$ equation.
\end{remark}

\subsection{Group (\ref{g3})}

This case can be reduced to the orbit of a point $p = (c, 1, 0)$, with $c > 0$, that is
\begin{equation*}
\alpha_p(s) = T_{3a}^s(p) = (c e^s, e^{-s}, a s)
\end{equation*}
and $[\alpha_p', \alpha_p'', \xi] = 2 c$ on $\R$. Then
\begin{equation*}
U_p(s) = \left( \frac{e^{-s} (-a+m)}{2 c},\frac{1}{2} e^s (a+m),1 \right)
\end{equation*}
and the associated immersion $\psi$ has coordinates
\begin{eqnarray*}
\psi_1 & = & \frac{1}{4} e^{s-t} \left( e^{2 t} \left( a + 2c + m \right) - a + 2c - m \right),  \\
\psi_2 & = & \frac{1}{4c} e^{-s-t} \left(e^{2 t} \left( - a + 2c + m \right) + a + 2c - m \right),  \\
\psi_3 & = & a s+\frac{e^{-2 t} \left(-a^2+(-2 c+m)^2\right)+e^{2 t} \left(a^2-(2 c+m)^2\right)+4 \left(a^2+4 c^2-m^2 \right) t}{16 c}.
\end{eqnarray*}

Now, the affine metric is
\begin{equation*}
h = \left( m \cosh(2 t) + \frac{\left(4 c^2 + m^2 - a^2\right) \sinh(2 t)}{4 c} \right) (ds^2 - dt^2)
\end{equation*}
and we get three different cases (see Figure 3):
\begin{enumerate}
\item If $4 c m \geq | 4 c^2 + m^2 - a^2 |$, then $\psi$ is an immersion on $\R^2$.
\item If $a = 0$ and $m > 2c$, then $\psi$ has an isolated singularity when $e^{2 t}(2c + m) = m - 2c$.
\item Otherwise, $\psi(\R^2)$ has a singular curve.
\end{enumerate}

\begin{figure}
\begin{center}
\includegraphics[width=0.3\textwidth]{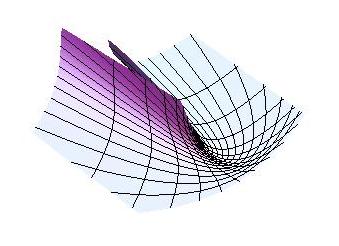} \qquad
\includegraphics[width=0.26\textwidth]{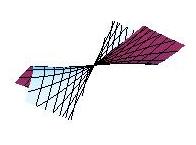} \qquad
\includegraphics[width=0.26\textwidth]{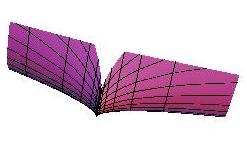} \\
\caption{(\ref{g3})-invariant indefinite improper affine spheres.}
\end{center}
\end{figure}

As consequence of the above study, we have proved the following result.

\begin{teo}
The only helicoidal indefinite improper affine spheres, with complete non flat affine metric, are the (\ref{g3})-invariant examples with $4 c m > | 4 c^2 + m^2 - a^2 |$.
\end{teo}

\section{Indefinite improper affine maps}
Helicoidal examples show the existence of an important amount of
indefinite improper affine spheres, glued by analytic curves where the affine
metric is degenerated, but the affine conormal $N = \Phi + \overline{\Phi}$ is well defined on the Lorentz surface $\Sigma$.

Hence, we say that $\psi : \Sigma \longrightarrow \R^3$ is an indefinite improper affine map, with constant affine normal $\xi$, if it admits the representation (\ref{inmer}) for a split-holomorphic curve $\Phi$ which satisfies that
$[\Phi+\overline{\Phi},\Phi_z,\overline{\Phi_z}]$ does not
vanish identically and $2 \esiz \Phi, \xi \esde = 1$.

A forthcoming study of this topic may be motivated by the following extension of Theorem \ref{t4}.

\begin{teo}
Let $\alpha : I \flecha \R^3$ be a regular analytic curve satisfying
\begin{equation*}
[\alpha'(s), \alpha''(s), \xi] \neq 0
\end{equation*}
for all $s \in I$. Then, there exists a
unique indefinite improper affine map $\psi$ containing $\alpha(I)$ in its set of singularities. Moreover,  the map $\psi$ can be written as (\ref{maximm}) in a neighborhood of $I$ in $\C'$, with
\begin{equation*}
\Phi = \frac{\alpha_z \por \alpha_{zz}} {2[\alpha_z, \alpha_{zz}, \xi]} + \frac{j}{2} \ \xi \por \alpha.
\end{equation*}
\end{teo}

\begin{eje}
The curve $\alpha : \R \flecha \R^3$ given by
\begin{equation*}
\alpha(s) = (\cos(s), \sin(s), \cos(2s)),
\end{equation*}
has $[\alpha', \alpha'', \xi] = 1$ and generates the improper affine map in Figure 4.
\end{eje}

\begin{figure}
\begin{center}
\includegraphics[width=0.5\textwidth]{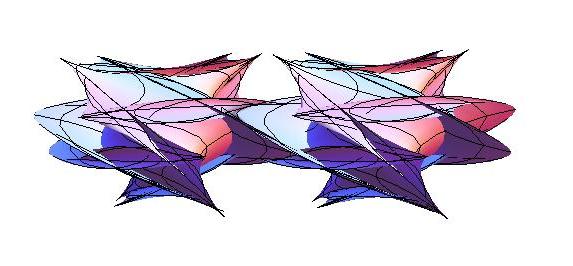}  \\
\caption{indefinite improper affine map.}
\end{center}
\end{figure}

\def\refname{References}

\end{document}